\documentclass{amsart}
\usepackage{amsmath, amscd, amssymb, amsthm}
\usepackage{bbm}
\usepackage{latexsym}
\usepackage{amsfonts}
\usepackage{graphicx}
\usepackage{xcolor}
\usepackage{needspace}
\definecolor{revieworange}{RGB}{224,112,0}
\definecolor{reviewred}{RGB}{180,35,35}
\definecolor{reviewblue}{RGB}{0,82,170}

\usepackage[all,cmtip]{xy}
\usepackage[colorlinks,linkcolor=blue,breaklinks=true,urlcolor=blue,citecolor=blue,anchorcolor=blue,pagebackref]{hyperref}%
\usepackage{geometry}
\newtheorem{theorem}{Theorem}
\newtheorem{lemma}{Lemma}

\newtheorem{proposition}{Proposition}

\newtheorem{conjecture}{Conjecture}

\newcommand{\R}{{\mathbb R}}

\newcommand{\C}{{\mathbb C}}

\renewcommand*\backref[1]{}
\renewcommand*\backrefalt[4]{ \ifcase #1 \or (cited on page #2) \else (cited on pages #2) \fi}

\newcommand{\be}{\begin{equation}}
\newcommand{\ee}{\end{equation}}
\newcommand{\bea}{\begin{eqnarray}}
\newcommand{\eea}{\end{eqnarray}}

\newcommand{\I}{\mathrm{i}}

\def\XXint#1#2#3{{\setbox0=\hbox{$#1{#2#3}{\int}$ }
\vcenter{\hbox{$#2#3$ }}\kern-.6\wd0}}

\begin{document}

\title[Balanced Bismut torsion-parallel manifold with constant holomorphic sectional curvature]{Balanced Bismut torsion-parallel manifold with constant holomorphic sectional curvature}

\author{Qingsong Wang}
\address{Qingsong Wang. Hal{\i}c{\i}o\u{g}lu Data Science Institute, University of California San Diego, La Jolla, CA 92093, USA}
\email{qswang92@gmail.com}
\thanks{Zheng is the corresponding author. He is partially supported by National Natural Science Foundations of China
with the grant No.12471039 and  12141101, and is supported by the 111 Project D21024.}

\author{Fangyang Zheng}
\address{Fangyang Zheng. School of Mathematical Sciences, Chongqing Normal University, Chongqing 401331, China}
\email{20190045@cqnu.edu.cn; franciszheng@yahoo.com}


\begin{abstract}
A long-lasting conjecture in non-K\"ahler geometry says that any compact Hermitian manifold with constant holomorphic sectional curvature must be either K\"ahler or Chern flat. The conjecture is known to be true in dimension 2 but remains open in  dimensions 3 or higher, except in several special cases. In this article we consider compact Hermitian manifolds that are Bismut torsion-parallel (BTP for brevity), which means that the Bismut connection has parallel torsion. For BTP manifolds that are not balanced, the conjecture was proved by Chen and Zheng. They also confirmed the conjecture for balanced BTP manifolds in dimension 3, utilizing a classification result for such threefolds by Zhao and Zheng. In a recent work, we confirmed the conjecture for balanced BTP manifolds in dimension 4 through a detailed case-by-case analysis. In this article, we establish the conjecture for balanced BTP manifolds in all dimensions.  We also disprove the full-rank $B$ conjecture by constructing a compact balanced BTP fivefold with full-rank $B$ that is neither K\"ahler nor Chern flat; its Chern holomorphic sectional curvature is nonconstant.
 \end{abstract}

\date{August 1, 2026}

\maketitle

\tableofcontents

\markleft{Qingsong Wang and Fangyang Zheng}
\markright{Balanced BTP manifolds with constant holomorphic sectional curvature}

\section{Introduction}\label{intro}

We begin with a brief discussion on the history of the following long-standing open question in non-K\"ahler geometry, which aims to understand the ``Hermitian space forms":

\begin{conjecture}[{\bf Constant Holomorphic Sectional Curvature Conjecture}] \label{conj1}
Any compact Hermitian manifold with constant holomorphic sectional curvature must be either K\"ahler or Chern flat.
\end{conjecture}

Given a Hermitian manifold $(M^n,g)$ and a non-zero complex tangent vector $X$ of type $(1,0)$, the holomorphic sectional curvature in the direction of $X$ is  defined by
$ H(X) = R_{X\bar{X}X\bar{X}}/|X|^4$,
where $R$ is the curvature tensor of the Chern connection $\nabla$ of $g$.

Recall that complete K\"ahler manifolds with constant holomorphic sectional curvature are called {\em complex space forms}. Their universal covers are the complex projective space ${\mathbb C}{\mathbb P}^n$, the complex Euclidean space ${\mathbb C}^n$, or the complex hyperbolic space ${\mathbb C}{\mathbb H}^n$, equipped with (a constant multiple of) the standard metrics. On the other hand, by the classical result of Boothby \cite{Boothby} in 1958, compact Chern flat manifolds are exactly compact quotients of complex Lie groups (equipped with left-invariant Hermitian metrics). Compact Chern flat manifolds can (and often) be non-K\"ahler when $n\geq 3$ (in the non-compact case there are examples of complete  non-K\"ahler Chern flat Hermitian surfaces \cite{YangZ}).

The conjecture states that any compact Hermitian manifold with $H=c$ must either be K\"ahler (thus a complex space form) or be Chern flat (thus a compact quotient of a complex Lie group). Note that the compactness assumption in the conjecture is a necessary one and there are counterexamples in the non-compact case (\cite{CCN}).

In complex dimension 2, Conjecture \ref{conj1} is known to be true. It started with the work of Balas and Gauduchon (\cite{BG}, \cite{Balas}) in 1985 for the $c\leq 0$ cases, and finished by the work of Apostolov, Davidov, and Muskarov \cite{ADM} in 1996 for the more challenging $c>0$ case.

For $n\geq 3$, the conjecture remains largely open, except in some special cases. The first substantial result towards the conjecture was obtained by Davidov, Grantcharov, and Muskarov \cite{DGM}, who showed among other things that the only twistor space with constant holomorphic sectional curvature is the complex space form ${\mathbb C}{\mathbb P}^3$. For the important special case of locally conformally K\"ahler manifolds, the conjecture was confirmed by the work of H. Chen, L. Chen and Nie \cite{CCN} when $c\leq 0$. The remaining $c>0$ case was resolved by Z. Huang and X. Wan in a recent work \cite{HW}. Besides twistor spaces and locally conformally K\"ahler manifolds, the conjecture was also confirmed for several special types of Hermitian manifolds, by the work of Tang \cite{Tang}, Zhou and Zheng \cite{ZhouZ}, Ma and Nie \cite{MN}, Li and Zheng \cite{LZ}, Rao and Zheng \cite{RZ}.

Bismut torsion-parallel (BTP) manifolds form an interesting class of special Hermitian manifolds. This class contains all Bismut flat manifolds (\cite{WYZ}) as special examples. More generally, all Bismut K\"ahler-like (BKL) manifolds and all Vaisman manifolds are also BTP. Recall that BKL means that the Bismut curvature tensor obeys all K\"ahler symmetries (\cite{ZZCrelle}, \cite{AOUV}), while Vaisman means those locally conformal K\"ahler manifolds whose Lee form is parallel under the Levi-Civita connection. Vaisman manifolds are known to be BTP by the work of Andrada and Villacampa \cite{AndradaV}.

Recall that a Hermitian metric $g$ is said to be {\em balanced} if $d(\omega^{n-1})=0$, where $\omega$ is the K\"ahler form of $g$ and $n$ is the complex dimension of the manifold.  Equivalently, balancedness means that the Chern-torsion trace vanishes.
For a BTP metric with constant Chern holomorphic sectional curvature, this
gives the Bismut--Ricci identity \eqref{eq:S} used below.

 BTP metrics are systematically studied in \cite{ZhaoZ24}. The majority of such manifolds are non-balanced, but there are also balanced ones in complex dimension 3 or higher. In \cite{ZhaoZ25}, Zhao and Zheng give a coarse classification result for compact balanced (but non-K\"ahler) BTP threefolds.

In \cite[Theorem 1.2]{ChenZ26}, S. Chen and Zheng confirmed Conjecture \ref{conj1} for non-balanced BTP manifolds in all dimensions, and for balanced BTP threefolds, utilizing the classification results of \cite{ZhaoZ25}. In \cite{WZ}, we confirmed the conjecture for all balanced BTP fourfolds, by a detailed case-by-case analysis.

The main purpose of this article is to confirm Conjecture \ref{conj1} for balanced BTP manifolds in all dimensions:

\begin{theorem} \label{thm1}
Let $(M^n,g)$ be a compact Hermitian manifold such that $g$ is balanced and Bismut torsion-parallel (BTP). If $g$ has constant Chern holomorphic sectional curvature, then it must be either K\"ahler or Chern flat.
\end{theorem}

We remark that in Theorem \ref{thm1} one can actually drop the compactness or completeness assumption, as the algebraic restriction of balanced BTP plus constant Chern holomorphic sectional curvature is already sufficiently restrictive. In other words we have the following slightly stronger statement:

\begin{theorem} \label{thm2}
Let $(M^n,g)$ be a  balanced BTP manifold with constant Chern holomorphic sectional curvature.  Then $g$ is either K\"ahler or Chern flat.
\end{theorem}

 Zhao and the second named author conjectured that a compact balanced
BTP metric with full-rank $B$ must be Chern flat; see Conjecture
\ref{conj2}. Section \ref{sec:full-rank-example} gives a compact balanced BTP
fivefold with full-rank $B$
which is neither K\"ahler nor Chern flat. Its Chern holomorphic sectional
curvature is nonconstant, showing that the constant holomorphic sectional
curvature hypothesis in Theorem \ref{thm2} is essential.

 In \cite{PodestaZ}, Podest\`a and Zheng showed that compact Chern flat BTP manifolds are exactly the compact quotients of reductive complex Lie groups.

The article is organized as follows. In \S 2, we will set up the notations and collect some known results in the literature that will be used later.  In \S 3 we will prove Theorem \ref{thm2} for the cases $c<0$ and $c=0$. In \S 4, we will deal with the case $c>0$ and $r_B<n$, where $r_B$ is the rank of the $B$ tensor.  \S 5 is devoted to the most challenging case of $c>0$ and $r_B=n$.   Section \ref{sec:full-rank-example} gives the counterexample to Conjecture \ref{conj2}.

\vspace{0.3cm}

\section{Preliminaries}

Let $(M^n,g)$ be a Hermitian manifold. For convenience, we will write the metric $g=\langle , \rangle$ as a symmetric complex bilinear form. Let $\nabla$, $\nabla^b$ be the Chern and Bismut connection of $g$, and denote by $T$, $T^b$ and $R$, $R^b$ the torsion and curvature of $\nabla$ and $\nabla^b$, respectively.

Let $e=\{ e_1, \ldots , e_n\}$ be a local unitary frame in $M^n$, namely, each $e_i$ is a local complex tangent vector field of type $(1,0)$, and $\langle e_i, \overline{e}_j\rangle =\delta_{ij}$ for all $1\leq i,j\leq n$. We will denote by $T^j_{ik}$ the Chern torsion components under the frame $e$, namely, $T(e_i,e_k)=\,\sum_{j=1}^nT^j_{ik} e_j$. As is well-known, $T(e_i,\overline{e}_k)=0$. Also, denote by $R_{i\bar{j}k\bar{\ell}}$ the components of the Chern curvature, and by $R^b_{ijk\bar{\ell}}$,  $R^b_{i\bar{j}k\bar{\ell}}$ the components of the Bismut curvature. When the metric $g$ is Bismut torsion-parallel (or BTP for brevity), which means $\nabla^bT^b=0$,  the results of Zhao and Zheng \cite{ZhaoZ24} give the following identities.

\begin{lemma} \label{lemma1}
Let $(M^n,g)$ be a Hermitian manifold. Then $g$ is BTP if and only if  $\nabla^bT=0$. In this case, under any local unitary frame $e$ and for all $1\leq i,j,k,\ell \leq n$  we have
\begin{eqnarray}
&& R^b_{ijk\bar{\ell}} \ = \ 0, \ \ \ \ \ \ R^b_{i\bar{j}k\bar{\ell}} \ = \ R^b_{k\bar{\ell}i\bar{j}} \,,  \label{eq:Rbsym}\\
&& \sum_{s} \big( T^{\ell}_{is} T^s_{jk} +  T^{\ell}_{js} T^s_{ki} + T^{\ell}_{ks} T^s_{ij}  \big) \ = \ 0, \label{eq:Jacobi} \\
&& R^b_{i\bar{j}k\bar{\ell}} - R_{i\bar{j}k\bar{\ell}} \ = \ \sum_s \big\{ T^{\ell}_{is} \overline{T^k_{js} } -  T^{s}_{ik} \overline{T^s_{j\ell } } - T^{j}_{is} \overline{T^k_{\ell s} } - T^{\ell}_{ks} \overline{T^i_{js} }  \big\}, \label{eq:Rb-R}\\
&&  R^b_{i\bar{j}k\bar{\ell}} - R^b_{k\bar{j}i\bar{\ell}} \ = \ \sum_s \big\{  T^{\ell}_{is} \overline{T^k_{js} } +  T^{j}_{ks} \overline{T^i_{\ell s} } -  T^{s}_{ik} \overline{T^s_{j\ell } } - T^{j}_{is} \overline{T^k_{\ell s} } - T^{\ell}_{ks} \overline{T^i_{js} }  \big\}. \label{eq:Rb-Rb}
\end{eqnarray}
\end{lemma}

Write $A_{i\bar{j}} = \sum_{s,t} T^s_{it} \overline{ T^s_{jt}}$, $B_{i\bar{j}} = \sum_{s,t} T^j_{st} \overline{ T^i_{st}}$. Then $A$ and $B$ are $(1,1)$ tensors globally defined on $M$, and both are Hermitian symmetric. When $g$ is BTP, we have $\nabla^bA=\nabla^bB=0$, so all eigenvalues of $A$ and $B$ are non-negative constants. Denote by $S_{i\bar{j}}=\sum_k R^b_{i\bar{j}k\bar{k}}$ the first Bismut Ricci tensor.

Using the metric $g$, we may regard $A$ as an endomorphism on the holomorphic tangent bundle $T^{1,0}M$ which maps $e_i$ to $\sum_j A_{i\bar{j}}e_j$ where $e$ is unitary. By an abuse of notation, we will denote this endomorphism still as $A$. Similarly, denote by $B$ (or $S$) the endomorphism of $T^{1,0}M$ corresponding to the $(1,1)$ tensor $B$ (or $S$).  The following commutation relation is proved in \cite[Proposition 1.5]{ZhaoZ24}.

\begin{lemma} \label{lemma2}
Let $(M^n,g)$ be a BTP manifold. Then  $A$   and  $B$   commute.
\end{lemma}

 Now let us assume that the BTP metric  $g$ also has constant Chern holomorphic sectional curvature $c$. This means that the symmetrization of the Chern curvature $R$ is a constant multiple of the curvature of the complex space form ${\mathbb C}{\mathbb P}^n$. This together with the properties (\ref{eq:Rb-R}) and (\ref{eq:Rb-Rb})  gives the following formula \cite[Lemma 8]{ChenZ26}:

\begin{lemma} \label{lemmaRb}
Let $(M^n,g)$ be a BTP manifold whose Chern holomorphic sectional curvature is equal to a constant $c$. Then under any local unitary frame $e$,
\begin{equation} \label{eq:Rb}
R^b_{i\bar{j}k\bar{\ell}} = \frac{c}{2} \big( \delta_{ij}\delta_{k\ell} + \delta_{i\ell}\delta_{kj} \big) - \frac{1}{2}\sum_{s=1}^n T^s_{ik} \overline{ T^s_{j\ell }} - \frac{3}{4}\sum_{s=1}^n \big( T^j_{is} \overline{ T^k_{\ell s}} + T^{\ell}_{ks} \overline{ T^i_{j s}}  \big) + \frac{1}{4}\sum_{s=1}^n \big( T^{\ell}_{is} \overline{ T^k_{j s}} + T^{j}_{ks} \overline{ T^i_{\ell s}}  \big),
\end{equation}
for any $1\leq i,j,k,\ell \leq n$.
\end{lemma}

When $g$ is balanced, which means $\sum_s T^s_{si}=0$ for all $i$, by letting $k=\ell$ and summing it up from $1$ to $n$, the above formula yields \cite[Lemma 2]{WZ}:

\begin{lemma}[\cite{WZ}] \label{lemmaS}
Let $(M^n,g)$ be a balanced BTP manifold with constant Chern holomorphic sectional curvature $c$. Then under any local unitary frame $e$, the first Bismut Ricci curvature tensor $S$ has components
\begin{equation} \label{eq:S}
S_{i\bar{j}} : = \sum_{k=1}^n R^b_{i\bar{j}k\bar{k}} = \frac{c}{2}(n+1) \delta_{ij} + \frac{1}{4}\big( B_{i\bar{j}} - A_{i\bar{j}}  \big), \ \ \ \ \ \ \ \forall \ 1\leq i,j\leq n.
\end{equation}
\end{lemma}

In the following, we will abbreviate {\em holomorphic sectional curvature} as HSC. In the balanced constant-HSC setting, Lemma \ref{lemma2} and
\eqref{eq:S} show that $A$, $B$, and $S$ commute pairwise. Thus they
can be diagonalized simultaneously by a local unitary frame.

We use the following stability identity, which is formula (1) of \cite{WZ}.

\begin{lemma}[\cite{WZ}] \label{lemma-stab}
Let $(M^n,g)$ be a BTP manifold. Then under any local unitary frame $e$, it holds for any $1\leq i,j,p, q,\ell \leq n$ that
\begin{equation}  \label{eq:stab}
 (R^b_{i\bar{j}}\cdot T)^{\ell}_{pq} := \sum_{s=1}^n \big\{ R^b_{i\bar{j}s\bar{\ell}}T^s_{pq} - R^b_{i\bar{j}p\bar{s}}T^{\ell}_{sq} - R^b_{i\bar{j}q\bar{s}}T^{\ell}_{ps} \big\} =0.
 \end{equation}
\end{lemma}

 In particular, if we take $i=j$ and sum it up from $1$ to $n$ in the above formula, and utilizing the symmetry property (\ref{eq:Rbsym}), then we get
\begin{equation}  \label{eq:stabS}
  \sum_{s} \big\{ S_{s\bar{\ell}}T^s_{pq} - S_{p\bar{s}}T^{\ell}_{sq} - S_{q\bar{s}}T^{\ell}_{ps} \big\} =0, \ \ \ \ \forall \ 1\leq p,q,\ell \leq n.
 \end{equation}
When $e$ makes $S$ diagonal this gives us the following, which is Lemma 6 of \cite{WZ}:

\begin{lemma}[\cite{WZ}] \label{lemmastabS2}
Let $(M^n,g)$ be a BTP manifold, and let $e$ be a local unitary frame under which $S$ is diagonal. Then we have
\begin{equation} \label{eq:stabS2}
(S_{i\bar{i}} + S_{k\bar{k}}-S_{j\bar{j}})\,T^j_{ik} = 0, \ \ \ \ \ \ \ \ \ \forall \ 1\leq i,j,k\leq n.
\end{equation}
\end{lemma}


\paragraph{\textbf{The pointwise Chern-torsion bracket.}}
Fix a point $x_0\in M$ and set $V:=T^{1,0}_{x_0}M$. Then
$V\cong\mathbb C^n$ is a complex vector space with Hermitian inner product
$(x,y):=\langle x,\overline y\rangle$. Since the Chern torsion $T$ is a
tensor, its value at $x_0$ defines a skew-symmetric complex-bilinear map
$T_{x_0}\colon\Lambda^2V\to V$. This construction uses only tangent vectors
at $x_0$ and does not involve extensions to local vector fields. Ni and Zheng
used this pointwise Lie-algebra interpretation of the Chern torsion in their
study of Chern Ambrose--Singer manifolds \cite{NiZ23}. Ni subsequently proved
that it applies to every Hermitian manifold with Bismut-parallel torsion and
used the notation $[\, ,\, ]_c$ \cite[Theorem 4.1]{Ni25}. Accordingly, define
\begin{equation}
[x,y]_c:=T_{x_0}(x,y),\qquad x,y\in V. \label{eq:Lie}
\end{equation}
Here the subscript $c$ stands for ``Chern'' and is unrelated to the constant
$c$ denoting the holomorphic sectional curvature. Equation
\eqref{eq:Jacobi} is precisely the Jacobi identity for $[\, ,\, ]_c$.
Consequently, $(V,[\, ,\, ]_c)$ is a complex Lie algebra, which we denote by
$\mathfrak g$. Moreover, $\nabla^bT=0$ implies that Bismut parallel transport
identifies these pointwise Hermitian Lie algebras by unitary Lie-algebra
isomorphisms. To distinguish this operation from the other brackets used in
the paper, we retain the subscript $c$ throughout. For subspaces
$U_0,Z_0\subseteq V$, set
\[
[U_0,Z_0]_c:=\operatorname{span}_{\mathbb C}
\{[u,z]_c:u\in U_0,\ z\in Z_0\}.
\]
Thus $[V,V]_c$ and $[W,W]_c$ are subspaces formed using the pointwise
Chern-torsion bracket, not the ordinary bracket of vector fields.

Following the argument in \cite{WZ}, let $r_B$ be the rank of $B$, let
$N=\ker(B)$, and put $W=N^\perp$. Then
$W=\operatorname{Im}T=[V,V]_c$; in particular, $\dim_{\C}W=r_B$.

\begin{lemma}[\cite{WZ}] \label{lemma-kernelB}
Let $(M^n,g)$ be a BTP manifold and $e$ a local unitary frame. Then we have
\begin{equation} \label{eq:kernelB}
 R^b_{i\bar{j}k\bar{\ell}} = 0,  \ \ \ \ \ \ \ \mbox{if} \ e_k\in W=(\mbox{ker}(B))^{\perp} \  \mbox{and} \ \,e_{\ell} \in N=\mbox{ker}(B),
\end{equation}
for any $1\leq i,j\leq n$.
\end{lemma}

Formula (7) in the proof of Lemma 4 of \cite{WZ} also gives the
corresponding statement for every pair of distinct eigenspaces of $B$.

 \begin{lemma} \label{lemma-BRb}
Let $(M^n,g)$ be a BTP manifold and  let  $e$   be  a local unitary frame under which $B$ is diagonal. Then whenever $B_{k\bar{k}} \neq B_{\ell \bar{\ell}}$ we have
\begin{equation*}  \label{eq:BRb}
R^b_{i\bar{j}k\bar{\ell}}=0, \ \ \ \ \ \forall \ 1\leq i,j\leq n.
\end{equation*}
\end{lemma}

 For any $u,v\in V$, let us denote by $D_{u\bar{v}}$ the endomorphism on $V$ defined by
$$D_{u\bar{v}}: V \rightarrow V, \ \ \ \ \ \ \ \ D_{u\bar{v}}(e_i) = \sum_{j=1}^n R^b_{u\bar{v}e_i\overline{e}_j} e_j,$$
where $e$ is any unitary frame. Clearly this definition is independent of the choice of $e$. Since $\overline{R^b_{v\bar{u}j\bar{i}} } = R^b_{u\bar{v}i\bar{j}} $, we know that the adjoint map $D_{u\bar{v}}^{\ast}$ (with respect to the Hermitian inner product on $V$) is equal to $D_{v\bar{u}}$. Note that under any unitary frame $e$, the matrix of $D_{u\bar{v}}$ is equal to $\Theta^b(u,\bar{v})$. By (\ref{eq:stab}), each $D_{u\bar{v}}$ is a derivation of the Lie algebra $V$, which will be called a {\bf  curvature derivation}. Its adjoint is also a derivation.

\begin{lemma} \label{lemma-DABS}
Let $(M^n,g)$ be a   balanced BTP manifold with constant Chern
holomorphic sectional curvature. For any $u,v\in V$, the curvature
derivation $D_{v\bar{u}}$ commutes with $A$, $B$, and $S$.
\end{lemma}

\begin{proof}
The BTP condition gives $\nabla^bA=\nabla^bB=0$, so every curvature
endomorphism of $\nabla^b$ commutes with $A$ and $B$. Balancedness enters
through the constant-HSC identity in Equation~\eqref{eq:S}, which gives
\[
 S=\frac{c}{2}(n+1)I+\frac14(B-A).
\]
Thus $S$ is a linear combination of $I$, $A$, and $B$, and every curvature
endomorphism commutes with $S$ as well.
\end{proof}




\vspace{0.3cm}

\section{The cases when $c\leq 0$}

In this section, we will prove Theorem \ref{thm2} in the cases $c<0$ and $c=0$. This will be achieved with the help of a square sum formula.

Fix a complex vector space $V\cong {\mathbb C}^n$ equipped with a
Hermitian inner product $(\cdot,\cdot)$; thus $(x,y)$ denotes the inner
product of $x,y\in V$. Denote by $\mbox{End}(V)$ the space of linear maps from $V$ into $V$. For $\phi \in \mbox{End}(V)$, its adjoint map $\phi^{\ast}$ is defined by
$ (\phi^{\ast}(x), y) = (x, \phi (y))$, $\forall \ x,y\in V$.
Let ${\mathcal C}$ be the space of all linear maps from $\Lambda^2V$ into $V$.
The Hermitian inner product $(\cdot,\cdot)$ on $V$ naturally induces a
Hermitian inner product on ${\mathcal C}$:
$$ (P, Q)_0 \ = \sum_{i,j,k=1}^n\!P^j_{ik} \overline{Q^j_{ik}}, \ \ \ \ \ \forall \ P, Q\in {\mathcal C}, $$
where $e$ is any unitary basis of $V$ and $P(e_i,e_k)=\sum_jP^j_{ik}e_j$ and $Q(e_i,e_k)=\sum_jQ^j_{ik}e_j$.  Clearly the quantity $(P, Q)_0$ is independent of the choice of the unitary basis $e$. We will also write $\parallel \! P\!\parallel^2_0 \ = (P,P)_0$.
For $\phi \in \mbox{End}(V)$, define the action $\pi (\phi)$ on ${\mathcal C}$ by
\begin{equation} \label{eq:action}
(\pi (\phi)P)(x,y) := \phi (P(x,y)) - P(\phi (x), y) - P(x, \phi (y)), \ \ \ \ \ \ \ \forall\ P \in {\mathcal C}, \, \forall \ x,y \in V.
\end{equation}

\begin{lemma}
The action defined by (\ref{eq:action}) satisfies the following
\begin{equation}  \label{eq:action3}
  (\pi(\phi))^{\ast} = \pi (\phi^{\ast}), \ \ \ \ \ [\pi(\phi), \pi(\psi)]= \pi([\phi, \psi]),  \ \ \ \ \ \pi (I)P=-P,
\end{equation}
for any $\phi$, $\psi \in \mbox{End}(V)$ and any $P \in {\mathcal C}$. Here $I$ is the identity map.
\end{lemma}
\begin{proof}
The second and third equalities follow directly from the definition (\ref{eq:action}), so we omit their proofs. To verify the first equality, let us take a unitary basis $e$ of $V$, and write $\phi (e_i) = \sum_j E_i^{j}e_j$. Then we have $\phi^{\ast}(e_i) = \sum_j \overline{E^i_j} e_j$. For $P\in {\mathcal C}$, write $P(e_i,e_k)=\sum_jP^j_{ik}e_j$ with $P^j_{ik}=-P^j_{ki}$. Then by (\ref{eq:action}) we have
$$ \big(\pi (\phi)Q\big)_{ik}^j = \sum_s \big( Q^s_{ik}E^j_s - E^s_iQ^j_{sk} - E^s_k Q^j_{is} \big).  $$
Similarly,
$$ \big(\pi (\phi^{\ast})P\big)_{ik}^j = \sum_s \big( P^s_{ik}\overline{E^s_j} - \overline{E^i_s}P^j_{sk} - \overline{E^k_s} P^j_{is} \big).  $$
From this we derive
$$ \sum_{i,j,k} P^j_{ik} \overline{\big(\pi (\phi)Q\big)_{ik}^j} = \sum_{i,j,k} \big(\pi (\phi^{\ast})P\big)_{ik}^j \overline{Q^j_{ik}}, $$
that is, $(P, \pi(\phi)Q)_0 = (\pi(\phi^{\ast})P, Q)_0$ for any $P, Q\in {\mathcal C}$, hence $(\pi (\phi))^{\ast} P = \pi(\phi^{\ast})P$. This completes the proof of the lemma.
\end{proof}

Next let $T:\Lambda^2V\rightarrow V$ be the element in ${\mathcal C}$ given by the Chern torsion. Let  $e$ be a unitary basis of $V$, and denote by $L_i$ the element in $\mbox{End}(V)$ defined by $L_i(e_k)=\sum_j T^j_{ik}e_j$. Its adjoint map is $L_i^{\ast}(e_k) = \sum_j \overline{T^k_{ij}}e_j$. We have
\begin{equation}
\sum_i L_i^{\ast} L_i = A, \ \ \ \ \ \sum_i L_iL_i^{\ast} = B.
\end{equation}

\begin{proposition} \label{prop1}
Let $(M^n,g)$ be a balanced BTP manifold with constant Chern holomorphic sectional curvature $c$. Then under any local unitary frame $e$, the Chern torsion tensor $T$ satisfies the equality
\begin{equation}  \label{eq:square}
\sum_i \parallel\! \pi(L_i^{\ast})T\!\parallel^2_0 \ \, = \ 2(n+1)c \parallel\! T\!\parallel^2_0.
\end{equation}
\end{proposition}

\begin{proof}
Note that by the definition (\ref{eq:action}) and the property (\ref{eq:Jacobi}), we always have $\pi(L_i)T=0$. The equation (\ref{eq:stabS}) says that $\pi (S)T =0$. Since $4S=2(n+1)c+B-A$ by (\ref{eq:S}), we conclude that $\pi(B-A)T = -2(n+1)c \,\pi(I)T = 2(n+1)cT$, where the last equality comes from the third identity in (\ref{eq:action3}). Therefore
\begin{eqnarray*}
\sum_i \parallel\! \pi(L_i^{\ast})T\!\parallel^2_0  & = & \sum_i \big( \pi(L_i^{\ast})T, \pi(L_i^{\ast})T\big)_0 \ = \ \sum_i \big( \pi(L_i)\pi(L_i^{\ast})T, T\big)_0 \\
& = & \sum_i \big( [\pi(L_i),\pi(L_i^{\ast})]T, T\big)_0 \ = \ \sum_i \big( \pi([L_i,L_i^{\ast}])T, T\big)_0 \\
& = &  \big( \pi ( \sum_i [L_i, L_i^{\ast}])T, T\big)_0 \ = \ \big( \pi(B-A)T, T\big)_0 \\
& = & \big( 2(n+1)cT, T\big)_0 \ = \ 2(n+1)c \parallel\!T\!\parallel^2_0.
\end{eqnarray*}
This completes the proof of the proposition.
\end{proof}

\begin{proof}[{\bf Proof of Theorem \ref{thm2} when $c\leq 0$.}]
Let $(M^n,g)$ be a balanced BTP manifold with constant Chern holomorphic sectional curvature $c$. Assume that $c\leq 0$. If $c<0$, then by (\ref{eq:square}) in Proposition \ref{prop1}, we get $\parallel\!T\!\parallel^2_0\ =0$, hence $T=0$ which means that the metric $g$ is K\"ahler. If $c=0$, then (\ref{eq:square}) implies that $\pi(L_i^{\ast})T=0$ for each $i$. For any $1\leq i,j,k,\ell \leq n$, we have by definition that
$$ \big(\pi(L_j^{\ast})T\big)^{\ell}_{ik} = \sum_s \big\{  \big(L_j^{\ast}\big)^{\ell}_sT^s_{ik} -  \big(L_j^{\ast}\big)^{s}_i T^{\ell}_{sk}- \big(L_j^{\ast}\big)^{s}_k T^{\ell}_{is} \big\} = \sum_s \big\{ T^s_{ik} \overline{T^s_{j\ell } } + T^{\ell}_{ks} \overline{T^i_{js } } -  T^{\ell}_{is} \overline{T^k_{js } }\big\}.$$
On the other hand, by (\ref{eq:Rb-R}) and (\ref{eq:Rb-Rb}), we get
\begin{eqnarray*}
R_{i\bar{j}k\bar{\ell}} - R_{k\bar{j}i\bar{\ell}} & = & -(R^b_{i\bar{j}k\bar{\ell}} - R_{i\bar{j}k\bar{\ell}}) + (R^b_{k\bar{j}i\bar{\ell}} - R_{k\bar{j}i\bar{\ell}}) + (R^b_{i\bar{j}k\bar{\ell}} - R^b_{k\bar{j}i\bar{\ell}} ) \\
& = & \sum_s \big\{ T^s_{ik} \overline{T^s_{j\ell } } + T^{\ell}_{ks} \overline{T^i_{js } } -  T^{\ell}_{is} \overline{T^k_{js } }\big\} \ \ = \ \ \big(\pi(L_j^{\ast})T\big)^{\ell}_{ik} \ \ = \ \ 0.
\end{eqnarray*}
Hence the metric $g$ is Chern K\"ahler-like in the sense that $R$ obeys all K\"ahler symmetries (\cite{YangZ}). In this case, by Tang's work \cite{Tang} we know that $g$ must be Chern flat since it has vanishing Chern holomorphic sectional curvature. This completes the proof of Theorem \ref{thm2} in the cases when $c<0$ and $c=0$.
\end{proof}

For compact BTP manifolds that are Chern flat, Podest\`a and Zheng showed in \cite{PodestaZ} that such manifolds are exactly the compact quotients of reductive complex Lie groups. To be more precise, if $(M^n,g)$ is a compact Chern flat BTP manifold (note that compact Chern flat manifolds are always balanced and are always quotients of complex Lie groups), then its universal cover is $(G,\tilde{g})$ where $G={\mathbb C}^k \times G_1 \times \cdots \times G_r$ with each $G_i$ being a simple complex Lie group, $k,r\geq 0$, and $\tilde{g}$ is the product metric $\tilde{g}=g_0\times g_1 \times \cdots \times g_r$, with $g_0$ the flat metric and $g_i$ a left-invariant Hermitian metric on $G_i$ for each $i$. Conversely, they also showed that any reductive complex Lie group admits left-invariant Hermitian metrics that are BTP and Chern flat, and on each simple factor, such metrics are unique up to Killing isometries. We refer the reader to \cite{PodestaZ} for more details.

\vspace{0.3cm}

\section{The case when $c>0$ and $r_B<n$}

The only remaining case in the proof of Theorem \ref{thm2} is $c>0$. Since the square-sum formula (\ref{eq:square}) does not yield an immediate conclusion, we divide the discussion into two parts: this section treats $r_B<n$, and the next treats $r_B=n$.

Throughout this section, we assume that $(M^n,g)$ is a balanced BTP manifold with constant Chern holomorphic sectional curvature $c>0$ and $B$-rank $r_B<n$. If $r_B=0$, then $B=0$, so $T=0$ and $g$ is K\"ahler. Hence we may assume that $0<r_B<n$.

Recall that we have the orthogonal decomposition $V=W\oplus N$ where $V$ is the holomorphic tangent space of $M^n$ at any given point, $N=\mbox{ker}(B)$ is the kernel space of $B$, and $W=N^{\perp}$ is the image space of $B$, which is also the image space of the skew linear map $T:\Lambda^2V \rightarrow V$. Let $S$ be the first Bismut Ricci curvature tensor. As mentioned before, we will use the same letter to denote a $(1,1)$ tensor on $M^n$ and the corresponding endomorphism on the holomorphic tangent bundle of $M^n$. Since $[S,B]=0$, $S$ will preserve the splitting $V=W\oplus N$, namely, $S(W)\subset W$ and $S(N)\subset N$. Write $S_W=S|_W$ and $S_N=S|_N$, then we have
$S=S_W \oplus S_N$.

By Lemma \ref{lemma2}  and \eqref{eq:S}, $A$, $B$, and  $S$ are
mutually commutative, so there exist local unitary frames $e$ such that
  \begin{equation} \label{eq:e}
\left\{ \begin{aligned} A, B, S \ \mbox{are all diagonal under} \ e; \hspace{1.35cm}\\
W = \mbox{span}\{ e_1, \ldots , e_{r_B}\}; \hspace{2.65cm}\\
S_W=\mbox{diag}(\lambda_1,\ldots , \lambda_{r_B}), \ \ \lambda_1\leq \cdots \leq \lambda_{r_B}. \end{aligned}
\right.
\end{equation}
Note that here we rearranged the order of $e_i$ so that $\lambda_1$ is the smallest eigenvalue of $S_W$ and $\lambda_{r_B}$ is the largest eigenvalue of $S_W$. Denote by $P=\{ 1, 2, \ldots , r_B\}$ and $Q=\{ r_B\!+\!1, \ldots , n\}$. We will use $i,j\in P$ and $\alpha , \beta \in Q$ for the index range. Introduce $A'$ and $A''$ on $N$ by letting
\begin{equation} \label{eq:A'A''}
 A'_{\alpha \bar{\beta}} = \sum_{i,k\in P}T^i_{\alpha k} \overline{ T^i_{\beta k}} , \ \ \ \ A''_{\alpha \bar{\beta}} = \sum_{i\in P, \gamma \in Q}T^i_{\alpha \gamma} \overline{ T^i_{\beta \gamma}} , \ \ \ \ \ \ \ \ \forall \ \alpha , \beta \in Q.
 \end{equation}
Since $N$ is the kernel of $B$, we have $T^{\alpha}_{\ast \ast}=0$ for any $\alpha \in Q$. So by definition $A_N=A'+A''$. By (\ref{eq:kernelB}) in Lemma \ref{lemma-kernelB}, we have $R^b_{p\bar{q}k\bar{\beta}}=0$ for any $k\in P$,  $\beta \in Q$, and any $p$, $q$. In particular, $R^b_{\alpha\bar{k}k\bar{\beta}}=0$ for any $k\in P$ and any  $\alpha ,\beta \in Q$. On the other hand, by (\ref{eq:Rb}) we have
\begin{equation} \label{eq:cross}
 R^b_{\alpha\bar{k}k\bar{\beta}} = \frac{c}{2}\delta_{\alpha \beta} +\frac{1}{2} \sum_{j\in P} T^j_{\alpha k} \overline{ T^j_{\beta k} } -\frac{3}{4} \sum_{j\in P} T^k_{\alpha j} \overline{ T^k_{\beta j} }  -\frac{3}{4} \sum_{\gamma \in Q} T^k_{\alpha \gamma } \overline{ T^k_{\beta \gamma } }.
\end{equation}
By summing up $k\in P$ in (\ref{eq:cross}) we get
\begin{equation}  \label{eq:A'+A''}
A' + 3A'' = 2cr_B I_N.
\end{equation}

\begin{lemma} \label{lemma-min}
Let $(M^n,g)$ be a balanced BTP manifold with constant Chern holomorphic sectional curvature $c>0$ and with $B$-rank $0<r_B<n$. Then we have
\begin{equation}  \label{eq:SN}
S_N = \frac{c}{2}(m+1)I_N + \frac{1}{2}A'',
\end{equation}
where $m=n-r_B$, and $A'$ and $A''$ are defined by (\ref{eq:A'A''}). In particular, $S_N $ is positive definite.
\end{lemma}

\begin{proof}
Since $B_N=0$, by (\ref{eq:S}) and (\ref{eq:A'+A''}) we have
\begin{eqnarray*}
 S_N &= & \frac{c}{2}(n+1)I_N -\frac{1}{4}A_N \ \ = \ \ \frac{c}{2}(n+1)I_N  -\frac{1}{4}(A'+A'') \\
 & = & \frac{c}{2}(n+1)I_N  -\frac{1}{4}(A'+3A'') + \frac{1}{2} A'' \ \ = \ \ \frac{c}{2}(n+1)I_N  -\frac{1}{4}2cr_BI_N + \frac{1}{2} A'' \\
 & = & \frac{c}{2}(n+1-r_B)I_N  + \frac{1}{2} A'' \ \ = \ \ \frac{c}{2}(m+1)I_N  + \frac{1}{2} A''.
 \end{eqnarray*}
 This completes the proof of the lemma.
\end{proof}

Since $A''\geq 0$, its trace $\operatorname{tr}(A'')\geq 0$.
  Indeed,
\begin{equation}\label{eq:traceAB}
\operatorname{tr}A
=\sum_{i,s,t}|T^s_{it}|^2
=\sum_{i,s,t}|T^i_{st}|^2
=\operatorname{tr}B.
\end{equation}
The middle equality only relabels the dummy indices $i$ and $s$.
Thus \eqref{eq:S} yields
$\operatorname{tr}S=\frac{c}{2}n(n+1)$. Taking traces in
\eqref{eq:SN},  we get
\begin{equation} \label{eq:traceSN}
\mbox{tr}(S_N) = \frac{c}{2}m(m+1) + \frac{\operatorname{tr}(A'')}{2}, \ \ \  \mbox{tr}(S_W) = \mbox{tr}(S) -\mbox{tr}(S_N)=\frac{c}{2}r_B(n+m+1) - \frac{\operatorname{tr}(A'')}{2} .
\end{equation}

Next let us estimate the lower bound of $S_W$.
\begin{lemma} \label{lemma-SWmin}
Let $(M^n,g)$ be a balanced BTP manifold with  constant  Chern
holomorphic sectional curvature  $c>0$ and with $B$-rank
$0<r_B<n$.   Set $m=n-r_B=\dim_{\C}N$. Then  $r_B\geq 2$, the smallest
eigenvalue of $S_W$ is simple, and
\begin{equation} \label{eq:SWmin}
 \lambda_{\min}(S_W)
 = (m+1)c+\frac{\operatorname{tr}(A'')}{m},
 \qquad
 \frac{2mc}{3}\leq\operatorname{tr}(A'')
 <\frac{r_B(r_B-1)mc}{m+2r_B}.
\end{equation}
 \end{lemma}

\begin{proof}
Choose a unitary frame $e$ satisfying  \eqref{eq:e}, and write
$\lambda_{\min}(S_W)=\lambda_1$. For $\alpha\in Q$ and  $j\in P$,
\eqref{eq:stabS2} with output index $1$ gives
\[
 (\lambda_\alpha+\lambda_j-\lambda_1)T^1_{\alpha j}=0.
\]
By \eqref{eq:SN}, $\lambda_\alpha>0$, while
$\lambda_j\geq\lambda_1$. Hence
\begin{equation}\label{eq:lemma12-min-output}
 T^1_{\alpha j}=0
 \qquad(\alpha\in Q,\ j\in P).
\end{equation}
Applying \eqref{eq:cross}  with $k=1$   therefore gives, for
$\alpha,\beta\in Q$,
\[
 0=\frac{c}{2}\delta_{\alpha\beta}
   +\frac{1}{2}\sum_{j\in P}
      T^j_{\alpha1} \overline{T^j_{\beta1}}
   -\frac{3}{4}\sum_{\gamma\in Q}
      T^1_{\alpha\gamma} \overline{T^1_{\beta\gamma}}.
\]
  Set
\[
  E=(T^1_{\alpha\beta})_{\alpha,\beta\in Q}\in M_m(\C),
 \qquad
 U=(T^j_{\alpha1})_{ \substack{\alpha\in Q\\2\leq j\leq r_B} }
       \in M_{m\times(r_B-1)}(\C),
\]
  where $U$ is the zero-column matrix when $r_B=1$. Then
\[
 \frac{3}{4}EE^\ast
 =\frac{c}{2}I_m+\frac{1}{2}UU^\ast
 \geq\frac{c}{2}I_m>0.
\]
Thus  the skew-symmetric matrix $E$ is  invertible.

For $\alpha,\beta\in Q$, Equation \eqref{eq:stabS2} also gives
\[
 (\lambda_\alpha+\lambda_\beta-\lambda_1)
 T^1_{\alpha\beta}=0.
\]
Since $S_N$ is diagonal in the chosen frame, this is
\[
S_NE+ES_N=\lambda_1E.
\]
Multiplication on the right by $E^{-1}$ and taking traces yield
\[
2\operatorname{tr}(S_N)=m\lambda_1.
\]
Using \eqref{eq:traceSN}, we obtain
\[
\lambda_1
 =\frac{2}{m}\operatorname{tr}(S_N)
 =(m+1)c+\frac{\operatorname{tr}(A'')}{m}.
\]

By Proposition 1 in \cite{WZ}, the case $r_B=1$ cannot occur. Hence
$r_B\geq2$.

It remains to prove that $\lambda_1$ is simple. Suppose instead that
 $\lambda_2=\lambda_1$.  Set
\[
E'=(T^2_{\alpha\beta})_{\alpha,\beta\in Q}\in M_m(\C),
 \qquad
 U'=(T^j_{\alpha2})_{\substack{\alpha\in Q\\j\in P\setminus\{2\}}}
       \in M_{m\times(r_B-1)}(\C).
\]
Because $\lambda_2=\lambda_1$,  the   same stability argument with output
index $2$ gives $T^2_{\alpha j}=0$ for every $\alpha\in Q$ and
$j\in P$. Equation \eqref{eq:cross} with $k=2$ then yields
\[
\frac{3}{4}E'E'^\ast
 =\frac{c}{2}I_m+\frac{1}{2}U'U'^\ast
 \geq\frac{c}{2}I_m>0.
\]
Hence $E'$ is also invertible.

Thus every unit vector in $\operatorname{span}\{e_1,e_2\}$ would
produce an invertible torsion matrix. We now show that the resulting matrix
pencil must contain a singular member.

For $t\in\C$, put $\rho=(1+|t|^2)^{1/2}$ and define
\[
\widetilde e_1=\frac{e_1+\overline t\,e_2}{\rho},
 \qquad
 \widetilde e_2=\frac{-t\,e_1+e_2}{\rho},
 \qquad
 \widetilde e_p=e_p\quad(3\leq p\leq n).
\]
This frame is unitary, preserves $W\oplus N$, and still diagonalizes
$S$ because $\lambda_1=\lambda_2$. If $\widetilde T$ denotes the
torsion components in this frame, then the output transformation gives
\[
\widetilde E(t)
 :=(\widetilde T^1_{\alpha\beta})_{\alpha,\beta\in Q}
 =\frac{E+tE'}{\rho}.
\]
Repeating the preceding positive-definite calculation with
$\widetilde e_1$ shows that $\widetilde E(t)$ is invertible for every
$t\in\C$. However,
\[
p(t)=\det(E+tE')
\]
has degree exactly $m$, since its leading coefficient is
$\det(E')\neq0$. As  $m>0$,  the polynomial $p$ has a complex zero,
contradicting the invertibility of $\widetilde E(t)$. Therefore
$\lambda_1<\lambda_2$.

Since $r_B\geq2$ and $\lambda_1$ is simple,
\[
\operatorname{tr}(S_W)>r_B\lambda_1.
\]
  Using \eqref{eq:traceSN}, the formula for $\lambda_1$, and $n=m+r_B$,
we obtain
\[
(m+1)c+\frac{\operatorname{tr}(A'')}{m}
 <\frac{c}{2}(n+m+1)-\frac{\operatorname{tr}(A'')}{2r_B}.
\]
Equivalently,
\[
\left(\frac{1}{m}+\frac{1}{2r_B}\right)\operatorname{tr}(A'')
 <\frac{c}{2}(r_B-1),
\]
which gives
\[
\operatorname{tr}(A'')<\frac{r_B(r_B-1)mc}{m+2r_B}.
\]

Finally, take $\alpha=\beta$  and $k=1$ in   \eqref{eq:cross}. Since
$T^1_{\alpha j}=0$ for $j\in P$,
\[
0=\frac{c}{2}
   +\frac{1}{2}\sum_{j\in P}|T^j_{\alpha1}|^2
   -\frac{3}{4}\sum_{\beta\in Q}|T^1_{\alpha\beta}|^2.
\]
Summing over $\alpha\in Q$ gives
\[
 3\sum_{\alpha,\beta\in Q}|T^1_{\alpha\beta}|^2
 =2mc+2\sum_{j\in P}\sum_{\alpha\in Q}|T^j_{\alpha1}|^2
 \geq2mc.
\]
Therefore
\[
\operatorname{tr}(A'')
 =\sum_{j\in P}\sum_{\alpha,\beta\in Q}
      |T^j_{\alpha\beta}|^2
 \geq\sum_{\alpha,\beta\in Q}|T^1_{\alpha\beta}|^2
 \geq\frac{2mc}{3}.
\]
This  proves \eqref{eq:SWmin}.
\end{proof}

Now we are ready to prove Theorem \ref{thm2} in the case when $c>0$ and $r_B<n$. We will divide the proof in two parts, depending on whether $[W,W]_c=0$ or $[W,W]_c\neq 0$. Recall that $V$ is the holomorphic tangent space of $M^n$ at any fixed point, the Chern torsion $T$ makes $V$ a complex Lie algebra, and $W=[V,V]_c$ is the commutator, with complex dimension $r_B$. So the condition $[W,W]_c=0$ means that $V$ is a $2$-step solvable Lie algebra. First let us deal with the $[W,W]_c=0$ case.

\begin{proposition}\label{prop2}
Let $(M^n,g)$ be a balanced BTP manifold   whose  Chern holomorphic
sectional curvature   is the positive constant $c$. Suppose that the
 $B$-rank satisfies $0<r_B<n$. Then $[W,W]_c=0$ cannot occur.
\end{proposition}

\begin{proof}
  {Set $m=n-r_B$, and choose } a local unitary frame   {$e$ satisfying
\eqref{eq:e}. Write
}\[
 {P=\{1,\ldots,r_B\},\qquad Q=\{r_B+1,\ldots,n\},
}\]
{and write $\lambda_a=S_{a\bar a}$. Since $W=\operatorname{Im}T$ } and
 $[W,W]_c=0$, we have
  \[
 {T^\alpha_{pq}=0
 \quad (1\leq p,q\leq n,\ \alpha\in Q),
 \qquad
 T^s_{ij}=0
 \quad (i,j\in P,\ 1\leq s\leq n).
}\]
  {By Lemmas \ref{lemma-min} and \ref{lemma-SWmin},
}\[
 {\lambda_\alpha\geq\frac{c}{2}(m+1)>0,\qquad
 r_B\geq2,\qquad \lambda_1<\lambda_2,
}\]
{and
}\[
 {\lambda_1=(m+1)c+\frac{\operatorname{tr}(A'')}{m},\qquad
 \frac{2mc}{3}\leq\operatorname{tr}(A'')<
 \frac{r_B(r_B-1)mc}{m+2r_B}.
}\]
{By \eqref{eq:lemma12-min-output},
$T^1_{\alpha j}=0$ for $\alpha\in Q$ and $j\in P$.
}

{Every curvature derivation } preserves $W$  and commutes with $S$ by
Lemma \ref{lemma-DABS}{. It therefore preserves the one-dimensional
minimum eigenspace $\C e_1$ } of $S_W${. Consequently,
}\[
 {R^b_{p\bar q1\bar j}=0
 \qquad (1\leq p,q\leq n,\ 2\leq j\leq r_B).
}\]
{For } $2\leq i,j\leq r_B$,  {formula \eqref{eq:Rb}, the condition
$[W,W]_c=0$, and \eqref{eq:lemma12-min-output} give
}\begin{align*}
{0
 }&{=R^b_{i\bar1 1\bar j} }\\
 &{=\frac{c}{2}\delta_{ij}
   -\frac{1}{2}\sum_sT^s_{i1}\overline{T^s_{1j}}}\\
 &{\quad-\frac{3}{4}\sum_s\left(
   T^1_{is}\overline{T^1_{js}}
   +T^j_{1s}\overline{T^i_{1s}}\right)}\\
 &{\quad+\frac{1}{4}\sum_s\left(
   T^j_{is}\overline{T^1_{1s}}
   +T^1_{1s}\overline{T^i_{js}}\right)}\\
 &{=\frac{c}{2}\delta_{ij}
   -\frac{3}{4}\sum_{\alpha\in Q}
   T^j_{\alpha1}\overline{T^i_{\alpha1}}.
}\end{align*}
  {Thus the matrix
}\[
 {F=(T^j_{\alpha1})_{\alpha\in Q,\ 2\leq j\leq r_B}
 \in M_{m,r_B-1}(\C)
}\]
  {satisfies
}\[
 {F^\ast F=\frac{2c}{3}I_{r_B-1}.
}\]
{In particular, $\operatorname{rank}F=r_B-1$, and hence
}\begin{equation}{\label{eq:rm}
 m\geq r_B-1.
}\end{equation}

{We next bound } the largest eigenvalue $\lambda_{r_B}$   {of } $S_W$.   {For
$\alpha\in Q$ } and  $j\in P$,   {Equation \eqref{eq:stabS2} gives
}\[
 {(\lambda_\alpha+\lambda_{r_B}-\lambda_j)T^j_{\alpha r_B}=0.
}\]
  {Since $\lambda_{r_B}\geq\lambda_j$ and $\lambda_\alpha>0$, it follows that
}\begin{equation}{\label{eq:prop2-max-input}
 T^j_{\alpha r_B}=0
 \qquad (\alpha\in Q,\ j\in P).
}\end{equation}
  {All components with output in $Q$ vanish; the components
$T^s_{r_Bi}$ with $i\in P$ vanish because $[W,W]_c=0$; and the components
$T^s_{r_B\alpha}$ vanish by skew-symmetry and
\eqref{eq:prop2-max-input}. Therefore
}\begin{equation}{\label{eq:prop2-Arr}
 A_{r_B\overline{r_B}}=\sum_{s,t=1}^n|T^s_{r_Bt}|^2=0.
}\end{equation}

{Set
}\[
 {a=\sum_{j\in P}\sum_{\alpha\in Q}|T^{r_B}_{\alpha j}|^2,
 \qquad
 b=\sum_{\alpha,\beta\in Q}|T^{r_B}_{\alpha\beta}|^2.
}\]
{Taking $k=r_B$ and $\beta=\alpha$ in \eqref{eq:cross}, using
\eqref{eq:prop2-max-input}, and summing over $\alpha\in Q$ gives
}\begin{equation}{\label{eq:prop2-ab}
 a+b=\frac{2mc}{3}.
}\end{equation}
{The definition of $B$ uses ordered input pairs. The $P\times P$
terms vanish, the two mixed blocks both contribute $a$, and the
$Q\times Q$ block contributes $b$. Hence
}\[
 {B_{r_B\overline{r_B}}
 =\sum_{p,q=1}^n|T^{r_B}_{pq}|^2
 =2a+b
 \leq2(a+b)=\frac{4mc}{3}.
}\]
{Together with \eqref{eq:S} and \eqref{eq:prop2-Arr}, this yields
}\begin{equation}{\label{eq:max1}
 \lambda_{r_B}
 =\frac{c}{2}(n+1)+\frac{1}{4}B_{r_B\overline{r_B}}
 \leq\frac{c}{6}(3n+3+2m).
}\end{equation}

{Define
}\[
{X=\sum_{i,j\in P}\sum_{\alpha\in Q}|T^i_{\alpha j}|^2.
}\]
{Summing the diagonal instance $\beta=\alpha$ of \eqref{eq:cross}
over } $k\in P$   {and $\alpha\in Q$, and interchanging the two $P$-indices
in one of the mixed sums, gives
}\begin{equation}{\label{eq:prop2-mixed-sum}
 0=\frac{cr_Bm}{2}+\frac{1}{2}X-\frac{3}{4}X-\frac{3}{4}\operatorname{tr}(A'')
  =\frac{cr_Bm}{2}-\frac{1}{4}X-\frac{3}{4}\operatorname{tr}(A'').
}\end{equation}
If   {$X=0$, then \eqref{eq:prop2-mixed-sum} gives
$\operatorname{tr}(A'')=2cr_Bm/3$. The strict upper bound in \eqref{eq:SWmin} would then
imply
}\[
{\frac{2}{3}<
 \frac{r_B-1}{m+2r_B},
}\]
{or equivalently
}\[
{2m+r_B+3<0,
}\]
which is   {impossible. Thus $X>0$.
}

{Choose } $i,j\in P$ and   {$\alpha\in Q$ such that
$T^i_{\alpha j}\neq0$. Equation \eqref{eq:stabS2} now gives
}\[
{\lambda_i=\lambda_j+\lambda_\alpha.
}\]
{Since $\lambda_{r_B}\geq\lambda_i$ and $\lambda_j\geq\lambda_1$, the
bounds from Lemmas \ref{lemma-min} and \ref{lemma-SWmin} imply
}\begin{align}
{\lambda_{r_B}
 }&{\geq\lambda_1+\lambda_\alpha \notag}\\
 &{\geq (m+1)c+\frac{2c}{3}
       +\frac{c}{2}(m+1)
 =\frac{c}{6}(9m+13).
 \label{eq:max2}
}\end{align}
{Combining \eqref{eq:max1} and \eqref{eq:max2} gives
}\[
{9m+13\leq3n+3+2m.
}\]
{Since $n=m+r_B$, this is
}\[
{4m+10\leq3r_B.
}\]
 On the other hand,  {\eqref{eq:rm} gives $3r_B\leq3m+3$. Thus the
assumption $[W,W]_c=0$ forces $m+7\leq0$, whereas
$m=n-r_B=\dim_{\C}N\geq1$. This contradiction proves that
$[W,W]_c=0$ cannot occur.}
\end{proof}

\begin{proposition}\label{prop3}
Let $(M^n,g)$ be a balanced BTP manifold with constant Chern
holomorphic sectional curvature $c>0$. If $0<r_B<n$, then
$[W,W]_c=0$.
\end{proposition}

\begin{proof}
Work at a fixed point and write $V=T^{1,0}M$, $W=[V,V]_c$, and
$W_1=[W,W]_c$. Suppose for contradiction that $W_1\ne0$.
The subspace $W_1$ is an ideal of $V$. Indeed, $W$ is an ideal, and the
Jacobi identity gives
\[
 [x,[y,z]_c]_c=[[x,y]_c,z]_c+[y,[x,z]_c]_c\in[W,W]_c
 \qquad (x\in V,\ y,z\in W).
\]
Equation \eqref{eq:stabS} is the derivation identity
\[
 S[x,y]_c=[Sx,y]_c+[x,Sy]_c.
\]
Thus $S$ preserves both $W$ and $W_1$. Formula \eqref{eq:S} shows that
$S$ is self-adjoint; hence it also preserves $N=W^\perp$ and
$W_2=W\cap W_1^\perp$. We may therefore choose a unitary eigenbasis
of $S$ adapted to
\[
 V=W_1\oplus W_2\oplus N.
\]
By Lemma \ref{lemma-min} and the minimum-eigenvalue identity in
Lemma \ref{lemma-SWmin}, every eigenvalue of $S$ is positive.

Let $\lambda$ be the largest eigenvalue of $S|_{W_1}$, and let
$U\subset W_1$ be its eigenspace. We first show that $[U,V]_c=0$. If
$u\in U$ and $v$ is an eigenvector of $S$ with eigenvalue $\mu$, then
$[u,v]_c\in W_1$ and
\[
 S[u,v]_c=[Su,v]_c+[u,Sv]_c=(\lambda+\mu)[u,v]_c.
\]
Since $\mu>0$ and $\lambda$ is the largest eigenvalue of $S|_{W_1}$,
this forces $[u,v]_c=0$. The eigenvectors of $S$ span $V$, so
$[U,V]_c=0$.

We next prove that $\dim U=1$.
Every curvature endomorphism $D_{x\bar y}$ is a derivation and hence
preserves $W_1$. Under the present hypotheses it also commutes with $S$
by Lemma \ref{lemma-DABS}. Consequently $D_{x\bar y}(U)\subset U$.

Set $d=\dim U$ and choose an $S$-adapted unitary basis
$e_1,\ldots,e_n$ such that $e_1,\ldots,e_d$ span $U$. For
$1\leq a,b\leq d$ and $d<p,q\leq n$, preservation of $U$ gives
$R^b_{p\bar a b\bar q}=0$. Since $[U,V]_c=0$,
 the constant-HSC curvature formula \eqref{eq:Rb} reduces to
\[
 0=R^b_{p\bar a b\bar q}
 =\frac{c}{2}\delta_{ab}\delta_{pq}
  -\frac{3}{4}\sum_{\gamma=d+1}^n
       T^a_{p\gamma}\overline{T^b_{q\gamma}}.
\]
For $1\leq a\leq d$, put
$F_a=(T^a_{pq})_{p,q=d+1}^n$. Then
\[
 F_aF_b^*=\frac{2c}{3}\delta_{ab}I_{U^\perp}.
\]
If $d>1$, choose $a\ne b$.
The space $U^\perp$ is nonzero because $U\subset W$ and $r_B<n$.
The identity for $a=b$ makes $F_a$ invertible, whereas
$F_aF_b^*=0$ gives $F_b=0$, contradicting
$F_bF_b^*=\frac{2c}{3}I_{U^\perp}$. Therefore $d=1$.

Retain an $S$-adapted unitary basis with $U=\C e_1$, and write
$\lambda=S_{1\bar1}$ and
$S'=S|_{U^\perp}=\operatorname{diag}(\lambda_2,\ldots,\lambda_n)$.
For $F=(T^1_{pq})_{p,q=2}^n$, the preceding identity gives
\[
 FF^*=\frac{2c}{3}I_{n-1}.
\]
Since $e_1$ is central, all torsion components having $e_1$ as an
input vanish. Therefore
\[
 B_{1\bar1}=\operatorname{tr}(FF^*)=\frac{2c}{3}(n-1),
 \qquad A_{1\bar1}=0.
\]
Formula \eqref{eq:S} now gives
\begin{equation}\label{eq:29}
 \lambda=\frac{c}{2}(n+1)+\frac{1}{4}B_{1\bar1}
 =\frac{c}{3}(2n+1).
\end{equation}
Because the basis diagonalizes $S$, Equation \eqref{eq:stabS2} with
output index $1$ is equivalent to
\[
 \lambda F=S'F+FS'.
\]
The matrix $F$ is invertible, so
$\lambda I_{n-1}=S'+FS'F^{-1}$. Taking traces yields
\[
 (n-1)\lambda=2\operatorname{tr}(S')
 =2\bigl(\operatorname{tr}S-\lambda\bigr).
\]
By \eqref{eq:traceAB}, taking the trace of \eqref{eq:S} gives
$\operatorname{tr}S=\frac{c}{2}n(n+1)$.
Hence
\begin{equation}\label{eq:30}
 \lambda=\frac{2}{n+1}\operatorname{tr}S=nc.
\end{equation}
Equations \eqref{eq:29} and \eqref{eq:30} imply
$nc=\frac{c}{3}(2n+1)$, so $n=1$. This contradicts $0<r_B<n$ and
proves the proposition.
\end{proof}

Combining Propositions \ref{prop2} and \ref{prop3}, we get the proof of Theorem \ref{thm2} in the case when $c>0$ and $r_B<n$.

\vspace{0.3cm}

\section{The case when $c>0$ and $r_B=n$}

In this section, we will prove Theorem \ref{thm2} in the case $c>0$ and $r_B=n$. We want to show that this case cannot occur. The full rank condition on the $B$ tensor means that $[V,V]_c=V$, namely, $V$ is a perfect Lie algebra, so it must have a semisimple part. First let us remark that $V$ itself cannot be semisimple. Assume on the contrary that $V$ is semisimple{. Equation \eqref{eq:stabS} says that } $S$  {is a derivation, and every derivation of a semisimple Lie algebra } is traceless. On the other hand,   {Equation~\eqref{eq:traceAB} and \eqref{eq:S} give $\operatorname{tr}S=\frac{c}{2}n(n+1)>0$, a contradiction. Hence } $V$ cannot be semisimple.

Let us denote by ${\mathfrak r}$ the solvable radical of $V$, and write $K={\mathfrak r}^{\perp}$ for its orthogonal complement. Our first goal is to show that $K$ is actually a Lie subalgebra of $V$.

\begin{proposition} \label{prop4}
Let $(M^n,g)$ be a balanced BTP manifold whose Chern
holomorphic sectional curvature is a positive constant $c$ and whose
$B$-rank is $r_B=n$.  If $\mathfrak r$ is the solvable radical of the
pointwise torsion Lie algebra $\mathfrak g$, then
$K=\mathfrak r^\perp$ is a Lie subalgebra of $\mathfrak g$.
\end{proposition}
\begin{proof}
Choose unitary bases $(k_\alpha)_{\alpha=1}^s$ of $K$ and
$(r_i)_{i=1}^m$ of $\mathfrak r$.  Decompose the bracket of two
vectors in $K$ as
\begin{equation}\label{eq:full-bracket-KK}
 [k_\alpha,k_\beta]_c
 =\sum_\gamma L^\gamma_{\alpha\beta}k_\gamma
  +C_{\alpha\beta},
 \qquad
 C_{\alpha\beta}=\sum_iC^i_{\alpha\beta}r_i\in\mathfrak r .
\end{equation}
Thus $K$ is a Lie subalgebra if and only if $C=0$.

Every derivation of $\mathfrak g$ preserves its solvable radical.
Since the adjoint of a curvature derivation is again a derivation,
each curvature derivation preserves both $\mathfrak r$ and
$K=\mathfrak r^\perp$.  Consequently,
\[
 R^b_{\ast\bar\ast\alpha\bar i}=0.
\]
Until stated otherwise, Roman indices refer to the basis of $\mathfrak r$
and Greek indices to the basis of $K$.

We first locate the values of $C$ inside the radical.  Since
$\mathfrak r$ is an ideal, $[\mathfrak r,K]_c\subseteq\mathfrak r$, and hence
$T^\alpha_{\beta i}=0$.
Recall that the Hermitian inner product on $V$ is
$(x,y)=\langle x,\overline y\rangle$. Thus the four terms in
\eqref{eq:Rb} containing a torsion component with output in $K$ and one
input in $\mathfrak r$ vanish. From
$0=R^b_{\beta\bar j\alpha\bar i}$ and \eqref{eq:Rb}, we obtain
\begin{equation}\label{eq:mixed-radical-row}
 0=-\frac{1}{2}\big(C_{\alpha\beta},[r_i,r_j]_c\big).
\end{equation}
Put
\[
 \mathfrak r_1=[\mathfrak r,\mathfrak r]_c,
 \qquad F=\mathfrak r\cap\mathfrak r_1^{\perp}.
\]
The subspace $\mathfrak r_1$ is an ideal of $\mathfrak g$, because
the Jacobi identity gives, for $x\in\mathfrak g$ and $u,v\in\mathfrak r$,
\[
 [x,[u,v]_c]_c=[[x,u]_c,v]_c+[u,[x,v]_c]_c\in\mathfrak r_1.
\]
Equation \eqref{eq:mixed-radical-row} now says that
\begin{equation}\label{eq:C-first-layer}
 C\in\Lambda^2K^*\otimes F.
\end{equation}

We now prove that this $F$-valued obstruction is exact and coclosed.

Let $\mathfrak q=\mathfrak g/\mathfrak r$.  The quotient maps identify
$K$ with $\mathfrak q$ and $F$ with
$\mathfrak r/\mathfrak r_1$.  Since $\mathfrak r$ acts trivially on its
abelianization, the adjoint action induces a representation of
$\mathfrak q$ on $F$, which in the chosen bases is
\[
 \bar M_\alpha f=P_F[k_\alpha,f]_c,\qquad f\in F.
\]
Transport the bracket of $\mathfrak q$ to $K$:
\[
 [k_\alpha,k_\beta]_K
 =\sum_\gamma L^\gamma_{\alpha\beta}k_\gamma.
\]
Let $\sigma\colon\mathfrak q\to K\subset\mathfrak g$ be the inverse of the
quotient map.  By the Levi decomposition theorem
\cite{JacobsonLieAlgebras}, there is also a Lie algebra section
$\tau\colon\mathfrak q\to\mathfrak g$.  Write
$\phi=\sigma-\tau\colon\mathfrak q\to\mathfrak r$ and
$\bar\phi=P_F\phi$.  For $F$-valued cochains, fix the sign convention
\[
 (d\psi)(x,y)
 =x\mathbin{\cdot}\psi(y)-y\mathbin{\cdot}\psi(x)-\psi([x,y]_K), \ \ \ \ \ \ x,y\in K.
\]
For $x,y\in\mathfrak q$, the obstruction associated with the section
$\sigma$ is
\[
 [\sigma x,\sigma y]_c-\sigma[x,y]
 =[\tau x,\phi(y)]_c-[\tau y,\phi(x)]_c
  -\phi([x,y])+[\phi(x),\phi(y)]_c.
\]
The first three terms represent $d\bar\phi(x,y)$ modulo
$\mathfrak r_1$, while the last term lies in $\mathfrak r_1$.  After passing to
$\mathfrak r/\mathfrak r_1$ and using its fixed representative space $F$,
\eqref{eq:C-first-layer} therefore gives
\begin{equation}\label{eq:C-coboundary}
 C=d\bar\phi.
\end{equation}

It remains to show that $C$ is coclosed.  Choose an adapted unitary basis
$V=K\oplus F\oplus\mathfrak r_1$; from now on, $i,j$ index $F$.
Every curvature derivation induces a derivation of the semisimple quotient
$\mathfrak q$.  Such a derivation is inner and traceless
\cite{JacobsonLieAlgebras}; hence
\begin{equation}\label{eq:quotient-trace-zero}
 {\mathcal L}_{\alpha\bar{i}}
 :=\sum_{\gamma=1}^s
 R^b_{\alpha\bar{i}\gamma\bar\gamma}=0.
\end{equation}
Give the cochain spaces their induced positive Hermitian inner products:
\[
 \big({P},{Q}\big)_{C^2(K;F)}=\frac{1}{2}\sum_{\alpha,\beta,i}
 P^i_{\alpha\beta}\overline{Q^i_{\alpha\beta}},
 \qquad
 \big({\psi},{\varphi}\big)_{C^1(K;F)}=\sum_{\alpha,i}
 \psi^i_\alpha\overline{\varphi^i_\alpha}.
\]
For these inner products, a direct calculation gives
\begin{equation}\label{eq:d-star}
 (d^*C)^i{}_{\alpha}
 =-\sum_{\beta,j}C^j_{\alpha\beta}
     \overline{(\bar M_\beta)^j{}_{i}}
  -\frac{1}{2}\sum_{\beta,\gamma}C^i_{\beta\gamma}
     \overline{L^\alpha_{\beta\gamma}}.
\end{equation}

We now expand the left-hand side of
\eqref{eq:quotient-trace-zero} directly.  Formula \eqref{eq:Rb},
\eqref{eq:C-first-layer}, and the ideality of $\mathfrak r$ give
\begin{equation}\label{eq:quotient-trace-reduced}
\begin{aligned}
 {\mathcal L}_{\alpha\bar{i}}
={}&\frac{1}{2}\sum_{\beta,j}C^j_{\alpha\beta}
       \overline{(\bar M_\beta)^j{}_{i}}
 -\frac{3}{4}\sum_\beta C^i_{\alpha\beta}
       \overline{\sum_\gamma L^\gamma_{\gamma\beta}}\\
 &+\frac{1}{4}\sum_{\beta,\gamma}C^i_{\beta\gamma}
       \overline{L^\alpha_{\beta\gamma}}.
\end{aligned}
\end{equation}
The constant term vanishes because $K\perp F$.  In the first quadratic
term, ideality and \eqref{eq:C-first-layer} leave only the displayed
$F$-sum; its positive sign uses
$T^j_{i\beta}=-T^j_{\beta i}=-(\bar M_\beta)^j{}_i$.
The two omitted quadratic sums contain a component
$T^\alpha_{i\ast}$ or $T^\gamma_{i\ast}$ and vanish because
$\mathfrak r$ is an ideal.

If $\bar k_\beta$ is the class of $k_\beta$ in $\mathfrak q$, then
\[
 \sum_\gamma L^\gamma_{\gamma\beta}
 =-\operatorname{tr}\!\left(
   \operatorname{ad}_{\bar k_\beta}\big|_{\mathfrak q}
  \right)=0.
\]
Thus \eqref{eq:quotient-trace-reduced} and \eqref{eq:d-star} give
\begin{equation}\label{eq:curvature-d-star}
 {\mathcal L}_{\alpha\bar{i}}
 =-\frac{1}{2}(d^*C)^i{}_{\alpha}.
\end{equation}
Equations \eqref{eq:quotient-trace-zero} and
\eqref{eq:curvature-d-star} give $d^*C=0$.
Combining this with \eqref{eq:C-coboundary} yields
\[
 \|C\|^2
 =\big(d\bar\phi,C\big)_{C^2(K;F)}
 =\big(\bar\phi,d^*C\big)_{C^1(K;F)}=0.
\]
Hence $C=0$.  By \eqref{eq:full-bracket-KK}, $K$ is a subalgebra, and its
projection to $\mathfrak g/\mathfrak r$ is an isomorphism.  Thus $K$ is a
Levi subalgebra.
\end{proof}

Now we are ready to finish the proof of Theorem \ref{thm2} in the case when $c>0$ and $r_B=n$.

\begin{proof}[{\bf Proof of Theorem \ref{thm2} when $c>0$ and $r_B=n$.}]

Retain the notation of Proposition~\ref{prop4}.  Thus
$\alpha,\beta$ range from $1$ to $s$, and $q$ ranges over the combined
orthonormal basis of $V=K\oplus\mathfrak r$.  The quotient trace of each
curvature derivation is zero, so
\begin{equation}\label{eq:double-quotient-trace-zero}
 \sum_{\alpha,\beta=1}^s
 R^b_{\alpha\bar\alpha\beta\bar\beta}=0.
\end{equation}
By (\ref{eq:Rb}) we get
\[ R^b_{\alpha\bar\alpha\beta\bar\beta} = \frac{c}{2}(1+\delta_{\alpha \beta}) +\sum_{q=1}^n \big\{ -\frac{1}{2} |T^q_{\alpha \beta}|^2 + \frac{1}{4} |T^{\alpha}_{\beta q}|^2 + \frac{1}{4} |T^{\beta}_{\alpha q}|^2 -\frac{3}{4} T^{\alpha}_{\alpha q} \overline{T^{\beta}_{\beta q} } - \frac{3}{4} T^{\beta}_{\beta q} \overline{T^{\alpha}_{\alpha q} } \big\}.
\]
We have $T^{\alpha}_{\beta i}=0$ since $[{\mathfrak r}, K]_c\subseteq {\mathfrak r}$, and $T^j_{\alpha \beta}=0$ since $[K,K]_c\subseteq K$ by Proposition \ref{prop4}.   {Hence the terms with $q\in\mathfrak r$ vanish. For $q=\gamma\in K$, dummy-index relabeling shows that the three squared-norm sums have the same total coefficient
$-\frac{1}{2}+\frac{1}{4}+\frac{1}{4}=0$. Therefore, after summing over
$\alpha,\beta\in K$, } we get
\[
 \sum_{\alpha,\beta=1}^s
 R^b_{\alpha\bar\alpha\beta\bar\beta}
 =\frac{c}{2}s(s+1)
 -\frac{3}{2}\sum_{\alpha,\beta,\gamma}T^\alpha_{\alpha \gamma}\overline{T^\beta_{\beta \gamma}} .
\]
Since $K$ is semisimple, $\mbox{tr}\big( \mbox{ad}_{k_{\gamma}}\big) =0$ in $K$, hence $\sum_{\alpha}T^{\alpha}_{\alpha \gamma} =0$. Therefore
\begin{equation*}
\sum_{\alpha,\beta=1}^s
 R^b_{\alpha\bar\alpha\beta\bar\beta}
 =\frac{c}{2}s(s+1).
\end{equation*}
From our assumption,  $c>0$. Moreover, $s=\dim K>0$: otherwise  $V$
  would be both solvable and perfect, which is impossible for the nonzero
Lie algebra $V$. The last equation therefore contradicts
\eqref{eq:double-quotient-trace-zero}. This shows that in Theorem
\ref{thm2} the case $c>0$ and $r_B=n$ cannot occur.
 \end{proof}

Combining the conclusions of \S 3-5, we get the proof of Theorem \ref{thm2}.

\vspace{0.3cm}

\section{A balanced BTP example with nonconstant holomorphic sectional curvature} \label{sec:full-rank-example}
{Motivated by the difficulty of the full-rank case, Zhao and the second
named author proposed the following conjecture \cite{ZhaoZprivate}.
}

\begin{conjecture}[{\bf Full-rank $B$ conjecture}]\label{conj2}
{Let $(M^n,g)$ be a compact balanced BTP manifold. If
$\operatorname{rank}B=n$, then $g$ is Chern flat.
}\end{conjecture}

{They were able to confirm the above conjecture up to $n=4$ \cite{ZhaoZprivate}. But as it turns out, the conjecture is false when $n\geq 5$, and below we will present a counterexample for $n=5$.  Recall that a Hermitian metric is Bismut-Ambrose-Singer (BAS) if its Bismut connection has parallel torsion and parallel curvature; hence every BAS metric is BTP. The following compact balanced BAS example disproves Conjecture
\ref{conj2}.
}

\begin{proposition}\label{prop:full-rank-example}
{There exists a compact locally homogeneous Hermitian fivefold
$(M_\Gamma,J,g)$ which is balanced and BAS (hence BTP), and such that
}\[
 {B=\operatorname{diag}\left(2,2,2,\frac{3}{2},\frac{3}{2}\right)>0
}\]
{in a unitary frame. The metric $g$ is neither K\"ahler nor Chern flat,
and its Chern holomorphic sectional curvature is nonconstant.
}\end{proposition}

\begin{proof}
{We first give the infinitesimal Hermitian data. Let $V=\C^5$ have unitary
basis
}\[
 {(e_1,e_2,e_3,u_1,u_2),\qquad
 U=\operatorname{span}_{\C}\{u_1,u_2\}.
}\]
{Define a complex skew-symmetric tensor $T:\Lambda^2V\to V$ by
}\[
 {T(e_1,e_2)=e_3,\qquad
 T(e_2,e_3)=e_1,\qquad
 T(e_3,e_1)=e_2,
}\]
\[
 {T(e_a,u)=\rho(e_a)u,\qquad T(u_1,u_2)=0,
}\]
{where
}\[
{\rho(e_1)=-\frac{\I}{2}
 \begin{pmatrix}0&1\\1&0\end{pmatrix},\quad
 \rho(e_2)=-\frac{\I}{2}
 \begin{pmatrix}0&-\I\\ \I&0\end{pmatrix},\quad
 \rho(e_3)=-\frac{\I}{2}
 \begin{pmatrix}1&0\\0&-1\end{pmatrix}.
}\]
{These brackets satisfy the Jacobi identity and define the Lie algebra
$\mathfrak{so}(3,\C)\ltimes\C^2$. This Lie algebra is perfect and hence
unimodular. Thus
}\[
{\sum_sT^s_{si}=0,
}\]
{which is the balanced condition. If $L_x(y)=T(x,y)$, direct calculation from
the definitions in }\S 2 {gives
}\[
{A=\sum_iL_{e_i}^*L_{e_i}
 =\operatorname{diag}\left(\frac{5}{2},\frac{5}{2},\frac{5}{2},
                            \frac{3}{4},\frac{3}{4}\right)
}\]
{and
}\[
{B=\sum_iL_{e_i}L_{e_i}^*
 =\operatorname{diag}\left(2,2,2,\frac{3}{2},\frac{3}{2}\right).
}\]
{Since a K\"ahler metric has $T=0$, whereas $B>0$ forces $T\ne0$, the
metric $g$ is not K\"ahler.
}

{We next prescribe the Bismut curvature. Let $P$ vanish on
$\operatorname{span}_{\C}\{e_1,e_2,e_3\}$ and be the identity on $U$. Put
}\[
{z_{11}=-\frac{\I}{2}e_3,\qquad
 z_{22}=\frac{\I}{2}e_3,\qquad
 z_{12}=-\frac{\I}{2}e_1+\frac{1}{2}e_2,\qquad
 z_{21}=-\frac{\I}{2}e_1-\frac{1}{2}e_2,
}\]
{and define complex-linear endomorphisms of $V$ by
}\[
{D_{e_a\bar e_b}=\operatorname{ad}_{T(e_a,e_b)},\qquad
 D_{e_a\bar u_p}=D_{u_p\bar e_a}=0,
}\]
\[
{D_{u_p\bar u_q}
 =\operatorname{ad}_{z_{pq}}+\frac{3}{4}\delta_{pq}P.
}\]
{Set $R^b(e_i,\bar e_j)|_V=D_{e_i\bar e_j}$ and let the $(2,0)$ and
$(0,2)$ parts of $R^b$ vanish. Extend $R^b$ to the complexification of the
underlying real space by reality and metric skew-adjointness. The associated
real Bismut torsion is determined by
}\[
{T^b(e_i,e_j)=-T(e_i,e_j),
}\]
\[
{T^b(e_i,\bar e_j)
 =\sum_k\left(T^j_{ik}\bar e_k-\overline{T^i_{jk}}e_k\right),
}\]
{together with skew-symmetry and conjugation.
}

{Substitution of the displayed matrices shows that $T^b$ is a three-form,
$R^b(X,Y)$ is skew-adjoint and complex linear, and the induced complex
structure is integrable. The same substitution gives
}\[
{\mathfrak S_{X,Y,Z}
 \left\{R^b(X,Y)Z-T^b\bigl(T^b(X,Y),Z\bigr)\right\}=0,
}\]
\[
{\mathfrak S_{X,Y,Z}R^b\bigl(T^b(X,Y),Z\bigr)=0,
\qquad
 R^b(X,Y)\cdot T^b=R^b(X,Y)\cdot R^b=0.
}\]
{Consequently the Nomizu bracket on
$\mathfrak l=\mathfrak h\oplus\mathfrak m$, where
$\mathfrak m$ is the underlying real space of $V$ and $\mathfrak h$ is the
span of the curvature endomorphisms, is
}\[
 [{A_1,A_2}]{=A_1A_2-A_2A_1,\qquad }[{A,X}]{=A(X),
}\]
\[
 [{X,Y}]{_{\mathfrak h}=-R^b(X,Y),\qquad
 }[{X,Y}]{_{\mathfrak m}=-T^b(X,Y).
}\]
{The preceding identities are precisely the Jacobi and invariance conditions
for this construction; see \cite[Theorem 3.7]{BarbaroP}. The canonical
connection preserves} $g$ {and $J$ and has skew torsion $T^b$. Since $J$ is
integrable, it is the Bismut connection. Hence the resulting homogeneous
Hermitian metric is BAS.
}

{The curvature algebra is
$\mathfrak h\cong\mathfrak{su}(2)\oplus\mathfrak u(1)$. A computation of the
Killing form identifies the transvection algebra as
}\[
 {\mathfrak l\cong
 \mathfrak{su}(3)\oplus\mathfrak{sl}(2,\C)_{\R}.
}\]
{An effective Lie group with this Lie algebra is
}\[
{G=\mathrm{PSU}(3)\times\mathrm{PSL}(2,\C).
}\]
{Let $H\subset G$ be the image of $U(2)$ under
}\[
{Q\longmapsto
 \left(
  \left[\operatorname{diag}\bigl(Q,\det(Q)^{-1}\bigr)\right],
  \left[(\det Q)^{-1/2}Q\right]
 \right).
}\]
{The brackets denote projective classes, and the second component is
independent of the square-root choice. The kernel is
$\mu_3=\{\zeta I_2:\zeta^3=1\}$, so
$H\cong U(2)/\mu_3\cong U(2)$. } In particular, {$H$ is compact and closed,
and $X=G/H$ has real dimension ten.
}

{Choose a closed orientable hyperbolic three-manifold $N$ and let
}\[
{\Gamma<\mathrm{PSL}(2,\C)
}\]
{be its holonomy group. Thus $\Gamma$ is a torsion-free uniform lattice.
Embed $\Gamma$ in the second factor of $G$ and define
}\[
{M_\Gamma=\Gamma\backslash G/H.
}\]
{Since $H$ is compact, the action of $\Gamma$ on $G/H$ is proper. The
stabilizer of $gH$ is $\Gamma\cap gHg^{-1}$; it is finite, and therefore
trivial because $\Gamma$ is torsion-free. Moreover,
}\[
{\Gamma\backslash G
 \cong\mathrm{PSU}(3)\times
       \bigl(\Gamma\backslash\mathrm{PSL}(2,\C)\bigr)
}\]
{is compact. Hence $M_\Gamma$ is a smooth compact Hermitian fivefold, and
the invariant BAS structure descends to it.
}

{Finally, the comparison identity \eqref{eq:Rb-R} gives
}\[
{R_{u_1\bar u_1e_1\bar e_1}=\frac{1}{4}\ne0.
}\]
{Thus the Chern connection is not flat. It also gives
}\[
{R_{e_1\bar e_1e_1\bar e_1}=0,\qquad
 R_{u_1\bar u_1u_1\bar u_1}=\frac{3}{4}.
}\]
{The Chern holomorphic sectional curvature is therefore not constant, and
$(M_\Gamma,J,g)$ is a counterexample to Conjecture \ref{conj2}.
}\end{proof}

\vspace{0.6cm}

\noindent\textbf{Acknowledgments.}
The second named author would like to thank Haojie Chen, Shuwen Chen, Xiaolan Nie, Kai Tang, Bo Yang, and Xiaokui Yang for their interest and helpful discussions. He is very grateful to Quanting Zhao for their long collaboration on BTP manifolds and for numerous discussions.

\vspace{0.3cm}

\noindent\textbf{Generative AI disclosure.}
During the development of this work, the authors used ChatGPT and Codex to assist with exploratory computations, possible proof directions, and checking algebraic reductions. The authors are responsible for the conceptualization and writing of the article, and take full responsibility for the correctness of the contents.

\vspace{0.3cm}

\noindent\textbf{Declaration on competing interests.}
All authors declare that there are no competing interests for this article.

\vspace{0.3cm}

\noindent\textbf{Added in proof.}
The preprint was finished on August 1, 2026 and was submitted to a math journal on August 4, 2026. It was also  submitted to arXiv on the same day, but was put `on hold' for more than a month. On August 19, H. Wang published \cite{HWang} on arXiv in which he independently proved the main result of this article (the $c\neq 0$ case of Theorem \ref{thm1}).


\vspace{0.3cm}

\begin{thebibliography}{99}


\bibitem {AndradaV}  A. Andrada and R. Villacampa, \emph{Bismut connection on Vaisman manifolds,} Math. Zeit. {\bf 302} (2022), 1091-1126.

\bibitem  {AOUV}  D. Angella, A. Otal, L. Ugarte, R. Villacampa,  \emph{On Gauduchon connections with K\"ahler-like curvature,}  Commun. Anal. Geom. \textbf{30} (2022), no.\,5, 961-1006.


\bibitem {ADM} V. Apostolov, J. Davidov, and O. Muskarov, \emph{Compact self-dual Hermitian surfaces,} Trans. Amer. Math. Soc. {\bf 348} (1996),  3051-3063.

\bibitem{Balas} A. Balas, \emph{Compact Hermitian manifolds of constant holomorphic sectional curvature, } Math. Zeit. {\bf 189} (1985),  193-210.

\bibitem{BarbaroP} {G. Barbaro and F. Pediconi,
\emph{On Bismut-Ambrose-Singer manifolds,}
arXiv: 2605.02485.
}

 \bibitem{BG} A. Balas and P. Gauduchon, \emph{Any Hermitian metric of constant nonpositive (Hermitian) holomorphic sectional curvature on a compact complex surface is K\"ahler, } Math. Zeit. {\bf 190} (1985), 39-43.




\bibitem {Boothby} W. Boothby, \emph{Hermitian manifolds with zero curvature,} Michigan Math. J. {\bf 5} (1958), no.2, 229-233.

 \bibitem {CCN} H. Chen, L. Chen, and X. Nie, \emph{Chern-Ricci curvatures, holomorphic sectional curvature and  Hermitian metrics,}  Sci. China Math. {\bf 64} (2021), 763-780.




\bibitem {ChenZ26} S. Chen and F. Zheng, \emph {Bismut torsion parallel metrics with constant holomorphic sectional curvature,} arxiv:2405.09110.





\bibitem{DGM} J. Davidov, G. Grantcharov, and O. Muskarov, \emph{Curvature properties of the Chern connection of twistor spaces,}
Rocky Mt. J. Math. {\bf 39} (2009), no.\,1,  27-48.

















\bibitem{HW} Z. Huang and X. Wan, \emph{Compact locally conformal K\"ahler manifolds with constant Chern holomorphic sectional curvature,} arXiv: 2606.24425

\bibitem{JacobsonLieAlgebras}
N. Jacobson.
\newblock {\em Lie Algebras}.
\newblock Dover Publications, New York, 1979.
\newblock Reprint of the 1962 Interscience edition.

\bibitem{LZ} Y. Li and F. Zheng, \emph{Complex nilmanifolds with constant holomorphic sectional curvature,}  Proc. Amer. Math. Soc. {\bf 150} (2022), 319-326.

\bibitem{MN} F. Ma and X. Nie, \emph{A remark on compact balanced threefolds with constant holomorphic sectional curvature,} Kodai Math. J. {\bf 47} (2024), no.\,1, 90-98.

\bibitem{Ni25} {L. Ni, \emph{Holonomy and the Ricci curvature of complex
Hermitian manifolds,} J. Geom. Anal. {\bf 35} (2025), Paper No.\,30.}

\bibitem{NiZ23} {L. Ni and F. Zheng, \emph{On Hermitian manifolds whose
Chern connection is Ambrose--Singer,} Trans. Amer. Math. Soc. {\bf 376}
(2023), no.\,9, 6681--6707.}


\bibitem{PodestaZ} F. Podest\`a, F. Zheng, \emph{A note on compact homogeneous manifolds with Bismut parallel torsion,} Sci. China Math. {\bf 68} (2025), no.\,7, 1643-1670.


\bibitem{RZ} P. Rao and F. Zheng, \emph{ Pluriclosed manifolds with constant holomorphic sectional curvature,} Acta. Math. Sinica (English Series) {\bf 38} (2022), no.6, 1094-1104.






 \bibitem {Tang} K. Tang, \emph{Holomorphic sectional curvature and K\"ahler-like metric} (in Chinese),  {Scientia Sinica Mathematica {\bf 51} (2021), no.\,12, 2013--2024.}



\bibitem {HWang} H. Wang, \emph{Bismut-torsion-parallel Hermitian manifolds with constant Chern holomorphic sectional curvature,}   arXiv: 2608.13386


\bibitem {WZ} Q. Wang and F. Zheng, \emph{Balanced Bismut torsion-parallel fourfold with constant holomorphic sectional curvature,}   arXiv: 2608.00457



\bibitem {WYZ} Q. Wang, B. Yang, and F. Zheng, \emph{On Bismut flat manifolds,} Trans. Amer. Math. Soc. {\bf 373} (2020), 5747-5772.

\bibitem {YangZ} B. Yang and F. Zheng, \emph{On curvature tensors of Hermitian manifolds,} Comm. Anal. Geom. {\bf 26} (2018), no. 5, 1193-1220.






\bibitem {ZZCrelle} Q. Zhao and F. Zheng, \emph{Strominger connection and pluriclosed metrics,}  J. Reine Angew. Math. (Crelles) {\bf 796} (2023),
245-267.



\bibitem {ZhaoZ24} Q. Zhao and F. Zheng, \emph{Curvature characterization of Hermitian manifolds with Bismut parallel torsion,} arXiv: {2407.10497}, to appear in Trans. Amer. Math. Soc.

\bibitem {ZhaoZ25} Q. Zhao and F. Zheng, \emph{On balanced Hermitian threefolds with parallel Bismut torsion,} arXiv:2506.15141

\bibitem {ZhaoZprivate} {Q. Zhao, private communication.}




\bibitem {ZhouZ}  W. Zhou and F. Zheng, \emph{ Hermitian threefolds with vanishing real bisectional curvature} (in Chinese), Scientia Sinica Mathematica {\bf 52} (2022), no.7, 757-764. English version: arXiv: 2103.04296

\end{thebibliography}
\end{document}